\documentclass[10pt]{amsart}
\usepackage{amsthm,amssymb,latexsym,amsmath}

\def \C {\mathbb{C}}

\def \fol {{\mathcal F}}
\def \F {{\mathcal F}}

\def \sing {{\rm Sing}}
\def \cod {{\rm Cod}}

\newtheorem{proposition}{Proposition}[section]
\newtheorem{definition}[proposition]{Definition}
\newtheorem{cor}[proposition]{Corollary}
\newtheorem{lema}[proposition]{Lemma}
\newtheorem{teo}[proposition]{Theorem}
\newtheorem{example}[proposition]{Example}
\newtheorem{re}[proposition]{Remark}
\begin{document}

\title[Classification of holomorphic  Foliations on Hopf manifolds]{Classification of holomorphic foliations on Hopf manifolds}

\author{Maur\'icio Corr\^ea }
\thanks{The first author is partially supported by CNPq grant number 300352/2012-3 and FAPEMIG grant number PPM-00169-13. The second author is partially supported by CAPES/Brazil.}
\address{\noindent Maur\' \i cio Corr\^ea\\
Departamento de Matem\'atica \\
Universidade Federal de Minas Gerais\\
Av. Ant\^onio Carlos 6627 \\
30123-970 Belo Horizonte MG, Brazil} \email{mauricio@mat.ufmg.br}

\author{Arturo Fern\'andez-P\'erez }

\address{Arturo Fern\'andez P\'erez \\
Departamento de Matem\'atica \\
Universidade Federal de Minas Gerais\\
Av. Ant\^onio Carlos 6627 \\
30123-970 Belo Horizonte MG, Brazil} \email{arturofp@mat.ufmg.br}

\author{Antonio M. Ferreira}

\address{Antonio Marcos Ferreira\\
Departamento de Matem\'atica \\
Universidade Federal de Minas Gerais\\
Av. Ant\^onio Carlos 6627 \\
30123-970 Belo Horizonte MG, Brazil} \email{ antoniomfs@gmail.com}


\begin{abstract}
We classify nonsingular holomorphic foliations of dimension and codimension one
on certain Hopf manifolds. We prove that any nonsingular codimension one distribution on an intermediary or generic Hopf manifold is integrable and admits a holomorphic first integral.
 Also, we prove  some results about singular holomorphic distributions on Hopf manifolds.

\end{abstract}
\maketitle

\section{Introduction and statement of results}
Let $W=\mathbb{C}^n-\{0\}$, $ n\geq 2$, and $f(z_1, z_2,...,
z_n)=(\mu_1 z_1,\mu_2 z_2,...,\mu_n z_n)$ be a diagonal contraction in $\mathbb{C}^n$, where $0<|\mu_i|<1$ for
all $1\leq i\leq n$. The space quotient
$X=W/<f>$ is a compact, complex manifold of dimension $n$ called of
Hopf manifold. The   geometry and topology of  Hopf manifolds have been studied by several authors, see for instance, Dabrowski \cite{Da}, Haefliger \cite{Hae}, Ise \cite{Ise}, etc.
\par In this paper, we are interested in the study of holomorphic foliations on Hopf manifolds.
 In \cite{Ma1}, using the Kodaira's classification of Hopf surfaces Daniel Mall obtained the classification of nonsingular holomorphic foliations on Hopf surfaces.  Motivated by this, we address the problem of classify nonsingular holomorphic foliations on dimension and codimension one on Hopf manifolds of dimension at least three.  We will consider the following types of Hopf manifolds.
\begin{definition}
We say that
\begin{enumerate}
\item $X$ is\textbf{ classical } if $\mu=\mu_1=\ldots=\mu_n$.
\item $X$ is \textbf{generic} if  $0<|\mu_1|\leq|\mu_2|\leq\ldots\leq|\mu_n|<1$ and there not exists non-trivial relation between the $\mu_i$'s in this way
$$\prod_{i\in A}\mu^{r_{i}}_{i}=\prod_{j\in B}\mu^{r_{j}}_{j},\quad r_i,r_j\in\mathbb{N},\quad A\cap B=\emptyset,\quad A\cup B=\{1,2,\ldots,n\}.$$
\item $X$ is \textbf{intermediary} if $\mu_1=\mu_2=\ldots=\mu_r$, where $2\leq r\leq n-1$ and there not exists non-trivial relation between the $\mu_i$'s in this way
$$\prod_{i\in A}\mu^{r_{i}}_{i}=\prod_{j\in B}\mu^{r_{j}}_{j},\quad r_i,r_j\in\mathbb{N},\quad A\cap B=\emptyset,\quad A\cup B=\{1,r+1,\ldots,n\}.$$
\end{enumerate}
\end{definition}
A line bundle $L$ on $X$ is the quotient of  $W\times \mathbb{C}$  by the
operation of a representation of the fundamental group
of  $X$,
$
  \varrho_{L}:   \pi_1(X)\simeq \mathbb{Z}  \longrightarrow   GL(1,\mathbb{C})  =  \mathbb{C}^*
$ in the following way
$$
\begin{array}{ccc}
  W\times \mathbb{C}& \longrightarrow &    W\times \mathbb{C}  \\
 (z,v)& \longmapsto  &     (f(x),  \varrho_{L}(\gamma)v)\
\end{array}
$$

We write $L=L_b$ for the bundle induced by the representation $
\varrho_{L}(\gamma)$ with  $b=\varrho_{L}(\gamma)(1).$ We prove the following theorem.

\begin{teo}\label{teo}
Let $X$ be a Hopf manifold, $\dim_{\mathbb{c}} X \geq 3$, and
$\mathcal{F}$ be a nonsingular one-dimensional holomorphic foliation in $X$ given
by a morphism   $ T_{\mathcal{F}}=L_b \rightarrow T_X$. Then the following holds:
\begin{itemize}
\item[(i)] If $X$ is classical, then  $b=\mu^{-m}$ with  $m \in \mathbb{N}$ and $m\geq
-1$. The foliation $\F$ is induced by a polynomial vector field
$$
g_1 \frac{\partial}{\partial
z_1}+\dots +g_n \frac{\partial}{\partial z_n},$$
where $g_i$ are homogeneous
polynomial of the same degree $m+1$, for all $1\leq i\leq n$, with
$\{g_1=\dots=g_n=0\}=\{0\}$.
\item[(ii)]
 If $X$ is generic, then   $b\in \{1, \mu_1, \dots ,\mu_n\}$. The foliation $\F$ is induced by a constant vector field.

\item[(iii)]
If $X$ is intermediary, then   $b\in \{1, \mu_1, \mu_{r+1},\mu_{r+2},\dots ,\mu_n\}$. We have the table

\begin{table}[htbp]
\centering
\begin{tabular}{|r|p{6cm}|}
\hline
$T_{\F}^*$& vector field inducing $\F$\\
\hline\hline
 $L_{1}$& $\sum\limits_{j=1} ^{r} g_j(z_1,\dots,z_r)\frac{\partial}{\partial z_{j}}+
\sum\limits_{k=r+1} ^{n}c^k z_k\frac{\partial}{\partial z_{k}}$ \\

\hline
 $L_{\mu_1}$& $c^1\frac{\partial}{\partial z_{1}}+\dots
+c^{r}\frac{\partial}{\partial z_{r}}$,  $c^{i}\neq 0$ for all $i$ \\
 \hline
 $L_{\mu_j}$ with $j>r$& $\frac{\partial}{\partial z_{j}}$\\
\hline
\end{tabular}
\end{table}
\end{itemize}
\end{teo}
 In the case that $X$ is a generic Hopf manifold, we have the following result.
\begin{teo}\label{generic_theorem}
All  holomorphic one-dimensional  foliations (possibly singular) on a
generic Hopf manifold  of dimension at least three are induced by monomial vector fields.
\end{teo}
\begin{definition}
Let $X=W/<f>$ be a Hopf manifold, where  $f(z_1,\dots ,z_n)=(\mu_1 z_1,\dots ,\mu_n z_n)$ is a contraction of $\mathbb{C}^n$, and $\mathcal{F}$
a nonsingular one-dimensional holomorphic foliation on $X$ given
by a morphism   $ T_{\mathcal{F}}=L_b \rightarrow
TX$. We say that $\mathcal{F}$ is \textbf{constant}  if $b=\mu_i$ for some $i=1,\dots ,n$; \textbf{ linear} if $b=1$, and \textbf{polynomial}  in otherwise.
\end{definition}
 It follows from \cite{Ma1} that a nonsingular holomorphic foliation on a Hopf surface has at least a compact leaf. The next result is a generalization of this fact, but only in the case of classical, intermediary or generic Hopf manifolds.

\begin{cor}\label{folhascompactas}
Let $X$ be a Hopf manifold, $\dim( X) \geq 3$, and
$\mathcal{F}$ be a nonsingular one-dimensional foliation in $X$. Then  $\mathcal{F}$ has a compact leaf. Moreover,
if $X$ is classical and $\F$ is a generic foliation (in the sense of \cite{Ad}) with  tangent bundle $T_{\F}=L_{\mu^{-m}}$, then $\F$ has
$$
\frac{m^n-1}{m-1}
$$
compact leaves.
\end{cor}
\begin{proof}
The proof will be divided into three parts:
 \\
\textbf{ Polynomial foliations.}
By Theorem \ref{teo}, there are polynomials foliations only in the classical case.
 Let $\F$ be a polynomial foliation on $X$. By Theorem \ref{teo}  the foliation
$\F$ is induced by a polynomial vector field $v=\sum\limits _{i=1} ^{n} g_i \frac{\partial}{\partial z_i}$ on $W$.
In this case we will consider the surjective morphism
$\alpha: {X}\rightarrow \mathbb{P}^{n-1}, \,\, \alpha(z)=[z]$,
whose fibers are  elliptic curves  $\mathbb{C}^*/<f>$.
The fiber  $\alpha^{-1}(z)$  is contained on a leaf of the  foliation, if and only if,
\begin{center}
$z_ig_j-z_jg_i=0,$ $\forall \  i,j=1,\dots,n.$
\end{center}
Consider the  map $\varphi: \mathbb{P}^{n-1}\rightarrow \mathbb{P}^{n-1}$ defined by
$\varphi(z)=[g_1(z):\ldots :g_n(z)]$.
This map always has fixed points (cf. \cite{atiyah}, pg 459).
Thus the fiber of $\alpha^{-1}(z)$ is a compact leaf of  the foliation. As we saw above,  the number of compact leaves of a generic polynomial vector field can be calculated by the number of fixed points of the correspondent polynomial   map. This follows from \cite[Proposition 4]{Ad}.
 \\
\textbf{ Constant foliations}. A constant foliation is induced by a vector field $v=\frac{\partial}{\partial z_i}$ for some $j=1,\dots,n,$ on $W$. The leaves of this foliation
in $W$ are an axis minus the origin, and planes parallel to this axis. Thus a constant foliation in $X$ has compact leaves.
 \\
\textbf{Linear foliations}. By Theorem \ref{teo}, linear foliations are induced by a vector field of the form
$$v=\sum\limits _{i=1} ^n g_i\frac{\partial}{\partial z_{j},}$$ where $g_i$ is linear polynomial, $1\leq i\leq n$. The vector field $v$ is complete in $\mathbb{C}^n$,
and the orbit of a point $z\in \mathbb{C}^n-\{0\}$ is diffeomorphic to $\mathbb{C}^*$  (cf.  \cite{book}, pg 23). Therefore, the foliation on $X$ has compact leaves.
\end{proof}
 Now, we present some results on codimension-one holomorphic distributions on Hopf manifolds.
\begin{teo}\label{teo1}
Let $X$ be a Hopf manifold, $\dim X \geq 3$, and
$\mathcal{F}$ be a nonsingular codimension-one distribution on $X$ given
by a morphism   $ \mathcal{N}^*_{\mathcal{F}}=L_b \rightarrow
\Omega^1 _X$. Then the following holds:
\begin{itemize}
\item[(i)] If $X$ is classical, then  $b^{-1}=\mu^{m}$ with  $m \in \mathbb{N}$ and $m\geq
1$. Furthermore $\fol$ is induced by a polynomial $1$-form
$$\omega=g_1dz_1+\dots +g_n dz_n,$$ where  $g_i$ are  homogeneous
polynomial of the same degree  $m-1$, for all $1\leq i\leq n$ with
$\{g_1=\dots=g_n=0\}=\{0\}$.
\item[(ii)]
 If $X$ is generic, then  $b^{-1}=\mu_j $ for some $j=1,2,\dots
,n$, and $\fol$ is induced by the $1$-form $\omega= dz_j$.

\item[(iii)] If $X$ is intermediary, then
$b^{-1}\in \{\mu_1, \mu_{r+1},\mu_{r+2},\dots ,\mu_n\}$.
The foliation  $\fol$ is induced by a constant  $1$-form.

\end{itemize}
\end{teo}
\par Note that Theorem \ref{teo1} implies that a distribution $\fol$ on an intermediary or generic Hopf manifold is induced by closed 1-form  $\omega$ in $\C^{n}-\{0\}$. In particular, this implies that  $\fol$ is integrable. We state it as follows.
\begin{cor}\label{cor43}
All nonsingular codimension-one holomorphic distributions on an intermediary
or generic Hopf manifold are integrable.
\end{cor}

\begin{example} In the  classical case the  Corollary \ref{cor43} is not true. Consider $X=\C^3-\{0\}/<f>$ a classical Hopf manifold of dimension $3$.
The $1$-form $\omega=y^p dx+ x^p dy + z^p dz$ induces on $X$ a non-integrable nonsingular distribution.

\end{example}
\par In the situation of Theorem \ref{teo1}, note that if $\fol$ satisfies the integrability condition we have the following corollary.

\begin{cor}\label{doze} Let $X$ be a Hopf Manifold, $\dim( X) \geq 3$, and
$\mathcal{F}$ be a nonsingular codimension-one foliation on $X$. Then   $\mathcal{F}$ has a holomorphic first integral.
\end{cor}
\begin{proof}
Assume that $\fol$ is induced by the 1-form $\omega$ on $\C^n-\{0\}$.
If $X$ is an intermediary or generic Hopf manifold, Theorem \ref{teo1}
implies that $\omega=dT$, where $T$ is a linear function and the proof ends. Now, if $X$ is classical, applying Theorem \ref{teo1}, we get  $$\omega=g_1dz_1+\dots +g_n dz_n,$$ where  $g_i$ are  homogeneous
polynomial of the same degree  $m-1$, for all $1\leq i\leq n$ with
$\{g_1=\dots=g_n=0\}=\{0\}.$ Since $n\geq 3$ and $\omega$ has an isolated singularity at $0\in\C^n$, then it follows from    Malgrange-Frobenius theorem \cite{book}  that   $\omega$ has a holomorphic first integral.
\end{proof}
 Theorem \ref{teo1} item $(i)$ together with Corollary \ref{doze} extends a  theorem due to  Ghys \cite{Ghys} for  nonsingular codimension-one foliations on classical Hopf manifolds. Moreover, in this work  we are not  supposing  that the distributions are integrable.
Now, in the special case of generic Hopf manifolds we have a more precise result.
\begin{teo}\label{prop6}
Any codimension-one holomorphic distribution (possibly singular) on a generic Hopf manifold of dimension at least three is integrable and induced by a monomial $1$-form.
\end{teo}

Finally we state some results about singular holomorphic distributions on Hopf manifolds.

\begin{teo}\label{singholomorphic}
Let $X$ be a Hopf manifold and  $\mathcal{F}$ be a holomorphic distribution of dimension or codimension one on $X$ with $\cod(\sing(\fol))\geq 2$. Then
\begin{enumerate}
\item if $n=2$ then  $Sing(\mathcal{F})=\emptyset $,
\item if  $n\geq 3$ then either  $Sing(\mathcal{F})=\emptyset$ or  $Sing(\mathcal{F})$ has at least a positive codimension component.
\end{enumerate}
\end{teo}

To prove Theorem \ref{singholomorphic}, we will use Residues theorems of Baum-Bott type (cf. \cite{Baum} and \cite{izawa}). Note that when $n=2$, the above result implies that there are not singular holomorphic foliations on $X$. Hence the classification of holomorphic foliations due by Mall \cite{Ma1} is complete. For codimension-one singular holomorphic foliations on classical Hopf manifolds we prove the following alternative.
\begin{teo}\label{Teo_brunella}
Let $\F$ be a singular codimension-one foliation on a  classical Hopf manifold of dimension at least three. Then
\begin{itemize}
 \item  either $\mathcal{F}$ has  an analytic invariant hypersurface,
 \item  or $\mathcal{F}$ has one dimensional subfoliation by elliptic   curves.
\end{itemize}
\end{teo}

We remark that Theorem \ref{Teo_brunella} should be viewed as an analogous version of \textit{Brunella - Conjecture}. More precisely, Marco Brunella conjectured that a codimension-one holomorphic foliation $\fol$ on $\mathbb{P}^n$, $n\geq 3$, satisfy the following alternative: either $\mathcal{F}$ has  an algebraic invariant hypersurface, or $\mathcal{F}$ has one-dimensional subfoliation by algebraic curves, see \cite{cerveau}.

\section{Cohomology of line bundles on Hopf manifolds}
\par Let $\Omega^{p}_{X}$ be the sheaf of germs of holomorphic $p$-forms
on a Hopf manifold $X$. Denote by $\Omega^{p}_{X}(L_b):=\Omega^{p}_{X}\otimes L_b$ and by $\pi:W\to X$ the natural projection on $X$. Consider a open covering $\{U_{i}\}$ of $X$ such that all sets open $U_{i}$ are Stein, simply-connected  and $\tilde{U}_{i}:=\pi^{-1}(U_{i})$ is a disjoint union of Stein open sets on $W$. Since $\pi$ is surjective, we have $A=\{\tilde{U}_{i}\}$ is open covering of $W$. It follows from the definition that
$$\displaystyle\tilde{U}_{i}=\cup_{r\in\mathbb{Z}}f^{r}(U_{i}).$$

\par Let $\varphi\in\Gamma(U_i,\Omega^{p}_{X}(L_b))$.  Then $\tilde{\varphi}=\pi^{*}(\varphi)$ belongs to $\Gamma(\tilde{U}_i,\pi^{*}(\Omega^{p}_{X}(L_b)))\cong\Gamma(\tilde{U}_i,\Omega^{p}_{W})$. Therefore we have a exact sequence of C\u{e}ch complexes
\begin{equation}\label{sequencia de Cech}
0\rightarrow \mathcal{C}^\textbf{.}(A, \Omega_{X} ^p(L_b))\stackrel{\pi^*}{\longrightarrow}
\mathcal{C}^\textbf{.}(A, \Omega^p _W) \stackrel{bId-f^*}{\longrightarrow} \mathcal{C}^\textbf{.}(A, \Omega_W ^p)\longrightarrow 0.
\end{equation}
From this we derive the long exact sequence of cohomology
 \begin{eqnarray*}
&&0\longrightarrow H^0(X, \Omega_{X} ^p (L_b))\longrightarrow H^0(W, \Omega_W ^p) \stackrel{p_0}{\longrightarrow} H^0(W, \Omega_W ^p) \longrightarrow  H^1(X, \Omega_{X} ^p (L_b))\rightarrow
\end{eqnarray*}		
where
$p_0=b\cdot Id-f^*: H^0 (W, \Omega_W ^1)\rightarrow H^0 (W, \Omega_W ^1)$
and $W=\mathbb{C}^{n}-\{0\}$.
D. Mall proved in  \cite{Ma} the following result.
\begin{teo}[Mall \cite{Ma}] \label{teocoh}
If $X$ is a Hopf manifold of dimension $n\geq 3$ and $L_b$ is a line bundle on $X$. Then
\begin{eqnarray*}\label{equacao99}
&&dim H^0(X, \Omega_{X} ^1(L_b))=dim H^0(X, \Omega_{X} ^{n-1}(L_b))=dim\,Ker(p_0)
\end{eqnarray*}
\end{teo}
\section{Holomorphic foliations}
Let $X$ be a complex manifold. A (nonsingular) \emph{foliation} $\F$, of
dimension $k$, on $X$   is a subvector bundle $T\F \hookrightarrow
T_X$, of generic rank $k$, such that $[T\F,T\F]\subset T\F$.

There is a dual point of view where $\mathcal F$ is determined by a
subvector bundle $N^*_{ \mathcal F}$, of   rank $n-k$, of the
cotangent bundle $\Omega^1_X = T^* X$ of $X$. The  vector bundle
$N_{\mathcal{F}}^*$ is called \emph{conormal vector bundle} of
$\fol$. The involutiveness   condition  is
replace by: if $d$ stands for the exterior derivative
then $dN_{\mathcal{F}}^* \subset N_{\mathcal{F}}^* \wedge
\Omega^1_X$ at the level of local sections.
The normal bundle $N_{ \mathcal F}$ of $\mathcal F$ is defined as
the dual of $N_{\mathcal{F}}^* $. We have the following exact
sequence
\[
0 \to T\mathcal F \to TX \to N_{\mathcal F} \to 0 \, .
\]

The $(n-k)$-th wedge product of the inclusion $N^*_\F
\hookrightarrow \Omega^1_X$ gives rise to a nonzero twisted
differential $(n-k)$-form $\omega\in H^0(X,\Omega^{n-k}_X\otimes
\mathcal{N})$ with coefficients in the line bundle
$\mathcal{N}:=\det(N_\F)$, which is \emph{locally decomposable} and
\emph{integrable}.
By construction the tangent bundle of a Hopf manifold $X$ is given
by
$$
TX=\bigoplus_{i=1}^n L_{\alpha_i},
$$
where $L_{\alpha_i}$ is the tangent bundle of the foliation induced
by the canonical vector field $\frac{\partial}{ \partial z_i}$.

\section{One-dimensional holomorphic foliations }
A nonsingular one-dimensional foliation $\mathcal{F}$ on a Hopf manifold $X$ is given by a line bundle $L_b:=T_{\mathcal{F}}$ on $X$
and an embedding $i:T_{\mathcal{F}}\rightarrow TX$, where $TX$ denotes the tangent bundle of $X$ and $b\in\mathbb{C}^*$. For guarantee the existence of
$\mathcal{F}$ such that $T_{\mathcal{F}}=L_b$ is necessary that $\dim \,H^0(X,TX\otimes L_{b^{-1}})>0$.
The tangent bundle $TX$ is isomorphic to $L_{\mu_1}\oplus\dots \oplus L_{\mu_n}$ and hence
\begin{equation*}
K_X\cong (L_{\mu_1}\otimes L_{\mu_2} \otimes \dots \otimes L_{\mu_n})^{*}\cong
(L_{\mu_1 \mu_2 \dots \mu_n})^*\cong L_{\mu_1 ^{-1}\mu_2 ^{-1}\dots \mu_n ^{-1}},
\end{equation*}
where $K_X$ is the canonical bundle of $X$. Let $L_a$ be a line bundle on $X$ with  $a\in \mathbb{C}^*$. We will  find a condition  on $a$  such that $\dim \,H^0(X,TX\otimes L_a)>0$. From the isomorphism $TX\cong\Omega_X ^{n-1} \otimes  K_X^*$ we get
$
H^0(X,TX\otimes L_a)   \cong H^0(X,\Omega_X ^{n-1} \otimes L_{\mu_1\mu_2\dots \mu_n a} )
$
Thus,
$
\dim H^0(X,TX\otimes L_a)>0$ if, and only if, $ \dim H^0(X,\Omega_X ^{n-1} \otimes L_{\mu_1\mu_2\dots \mu_n a} ) >0.$
\begin{lema}\label{lema11} Let $X$ be a classical, intermediary or generic Hopf manifold of dimension $n\geq3$, and let $L_b$ be a line bundle on $X$, with $b \in \mathbb{C}^*$.
\begin{enumerate}
\item[(i)] If $X$ is classical then $\dim H^0(X,\Omega_X ^{n-1} \otimes L_b)>0$ if, and only if, $b=\mu^m$, with $m \in \mathbb{Z}, \, m\geq n-1$.
\item[(ii)] If $X$ is generic then $\dim H^0(X,\Omega_X ^{n-1} \otimes L_b)>0$ if, and only if,  $b=\mu_1 ^{m_1} \mu_2^{m_2}\dots \mu_n ^{m_n},$ where $\mu_j \in \mathbb{N}$, and there exists $j_0\in \{1,\dots ,n\}$, such that $\mu_{j_0}\geq 0$, with $\mu_j \geq 1$
for all $ j\in \{1,\dots ,n\}\setminus\{j_0\}$.
\item[(iii)] If $X$ is intermediary then $\dim H^0(X,\Omega_X ^{n-1} \otimes L_b)>0$ if, and only if, $$b=\mu_1^{m}\mu_{r+1}^{m_{r+1}}\dots \mu_n ^{m_n},$$ with $m\geq r-1$ and $m_j\geq 1$ for all $j\geq r+1$ or $m\geq r$, and there exists $j_0\geq r+1$ with $m_{j_0}\geq 0$ and $m_j \geq 1$ for all $j \in \{r+1, r+2,\dots ,n\}\setminus\{j_0\}$.
\end{enumerate}
\end{lema}
\begin{proof}
By Theorem \ref{teocoh} we have $\dim H^0(X, \Omega_X ^{n-1}\otimes L_b)=\dim(ker\, p_0)$, where
\begin{center}
$p_0: H^0 (W, \Omega_ W ^{n-1})\longrightarrow H^0 (W, \Omega_ W ^{n-1})$, \quad$p_0=b Id-f^*$ \, and \, $W=\mathbb{C}^n-\{0\}$.
\end{center}
Let $\omega \in H^0(W,\Omega_W ^{n-1})$, then
$\omega= \sum\limits _{i=1} ^{n} g_i dz_1\wedge\dots \wedge  {dz_{i-1}}\wedge dz_{i+1}\wedge\dots \wedge dz_n$.  It  follows from Hartogs extension theorem that each $g_i$ can be represented by its Taylor series
\begin{center}
$g_i(z_1,z_2,\dots ,z_n)=\sum\limits _{\alpha \in \mathbb{N}^n} c_\alpha ^i z_1^{\alpha_1}z_2^{\alpha_2}\dots z_n ^{\alpha_n}$, for all $i=1,\dots ,n.$
\end{center}
Hence
\begin{equation}\label{equacao123}
p_0(\omega)=\sum\limits _{i=1} ^{n} \sum\limits _{\alpha \in \mathbb{N}^n}c_\alpha ^i (b-\mu_1 ^{\alpha_1+1}\dots \mu_n ^{\alpha_n+1}\mu_i ^{-1}) z_1^{\alpha_1} \dots z_n ^{\alpha_n}\widehat{dz_i},
\end{equation}
where $ \widehat{dz_i}:=dz_1\wedge\dots \wedge  {dz_{i-1}}\wedge dz_{i+1}\wedge\dots \wedge dz_n.$
First we consider the classical case. In this case $\mu_1=\dots =\mu_n=\mu$ and
\begin{equation*}
p_0(\omega)=\sum\limits _{i=1} ^{n} \sum\limits _{\alpha \in \mathbb{N}^n}c_\alpha ^i (b-\mu^{\alpha_1+\dots +\alpha_n+n-1}) z_1^{\alpha_1} \dots z_n ^{\alpha_n}  \widehat{dz}^i .\end{equation*}
so that $\dim(ker\,p_0) >0$ if, and only if,
$b=\mu^{m}$, for some $m \in \mathbb{N}$, $m\geq n-1$.
For the generic case, since $\mu_i$ has no relations, it  follows from (\ref{equacao123}) that $\dim(ker\,p_0) >0$ if, and only if,  $b=\mu_1 ^{m_1} \mu_2^{m_2}\dots \mu_n ^{m_n}$ where $\mu_j \in \mathbb{N}$, and there exists $j_0$ such that $\mu_{j_0}\geq 0$ e $\mu_j\geq 1\, \forall j\in \{1,\dots ,n\}\setminus\{j_0\}$.
Finally, for the intermediary case, we have $\mu_1=\dots =\mu_r=\mu$, so that
\begin{equation*}
p_0(\omega)=\sum\limits _{i=1} ^{n} \sum\limits _{\alpha \in \mathbb{N}^n}c_\alpha ^i (b-\mu ^{\alpha_1+\dots +\alpha_r+r} \mu_{r+1}^{\alpha_r+1}\dots \mu_n ^{\alpha_n+1}\mu_i ^{-1}) z_1^{\alpha_1}\dots z_n ^{\alpha_n}
\widehat{dz}^i
\end{equation*}

Since $\mu, \mu_{r+1},\dots ,\mu_n$ has no relations, we have
$\dim(ker\,p_0)>0$ if, and only if,
 $b=\mu^{m}\mu_{r+1}^{m_{r+1}}\dots \mu_n ^{m_n}$, with $m\geq r-1$ and $m_j\geq 1$ for all $j\geq r+1$, or $b=\mu_1^{m}\mu_{r+1}^{m_{r+1}}\dots \mu_n ^{m_n}$, with $m\geq r$, and there exists $j_0\geq r+1$ with $m_{j_0}\geq 0$ and $m_j\geq 1$ for all $j \in \{r+1, r+2,\dots ,n\}\setminus\{j_0\}$.
\end{proof}
The above lemma implies the following proposition.
\begin{proposition}\label{propo_lema}
Let $X$ be as in Lemma \ref{lema11} and $L_{a}$ be a line bundle on $X$. The following statements holds:
\begin{enumerate}
\item[(i)] If $X$ is classic then $\dim H^{0}(X,TX\otimes L_a)>0$ if, and only if, $a=\mu^{d}$, where $d\in \mathbb{Z}$ with $d\geq -1$.
\item[(ii)] If $X$ is generic then $\dim H^{0}(X,TX\otimes L_a)>0$ if, and only if, $$a=\mu^{d_{1}}\mu^{d_{2}}_{2}\ldots\mu^{d_{n}}_{n}$$ where there exists $j_{0}\in\{1,\ldots,n\}$ with $d_{j_{0}}\geq -1$ and $d_{j}\geq 0$ for all $j\in\{1,\ldots,n\}\setminus\{j_{0}\}$.
\item[(iii)] If $X$ is intermediary then $\dim H^{0}(X,TX\otimes L_a)>0$ if, and only if, $a=\mu^{d}\mu^{d_{r+1}}_{r+1}\ldots\mu^{d_{n}}_{n}$ where $d\geq -1$ and $d_{j}\geq 0$ for all $j\geq r+1$ or $d\geq 0$ and there exists $j_0\geq r+1$ with $d_{j_{0}}\geq -1$ and $d_{j}\geq 0$ for all $j\in\{r+1,r+2,\ldots,n\}\setminus\{j_{0}\}$.
\end{enumerate}

\end{proposition}
\begin{proof}
Since $
H^0(X,TX\otimes L_a)  \cong  H^0(X,\Omega_X ^{n-1} \otimes L_{\mu_1\mu_2\dots \mu_n a} )$ the proposition follows from  Lemma \ref{lema11}.
\end{proof}

\subsection{ Proof Theorem \ref{teo}}
The morphism $T_{\mathcal{F}}=L_b \rightarrow TX$ gives rise to a section \mbox{$s\in H^0(X,TX\otimes L_{b^{-1}})$}. On the other hand,  we have the isomorphism
$$\displaystyle TX\otimes L_{b^{-1}} \cong (W \times \mathbb{C} ^n )/(f \times A b^{-1}),$$ where
 $W=\mathbb{C}^n \setminus \{0\}$ and
$A  =\left(\mu_1,\dots,\mu_n \right)
.$
Therefore, a section \mbox{$s\in H^0(X,TX\otimes L_{b^{-1}})$} correspond to a section $\tilde{s} \in H^0(W, \mathcal{O}_W ^n)$, say $\tilde{s}=(g_1,\dots ,g_n)$, where
$g_i \in \mathcal{O}_W$ satisfying:
\begin{center}
$g_k(\mu_1 z_1,\dots ,\mu_n z_n)= \mu_k b^{-1} g_k(z_1,\dots ,z_n)$ for all $k=1,\dots ,n$.
\end{center}
It follows from Hartogs extension theorem that $\tilde{s}$
can be represented by its Taylor series

\begin{center}
$g_k(z_1,\dots ,z_n)=\sum{c_\alpha ^k z^\alpha}$ , $\alpha =(\alpha_1, \alpha_2, \dots ,\alpha_n) \in \mathbb{N}^n$, $k=1,2,\dots ,n$.
\end{center}
Then for all $k=1,\dots ,n$, we have
\begin{equation}\label{eq1}
c_{\alpha}^k \mu_1 ^{\alpha_1} \mu_2 ^{\alpha_2}\dots  \mu_ n ^{\alpha_n}=c_{\alpha}^k \mu_k b^{-1},
\textrm{ where } \alpha =(\alpha_1, \alpha_2, \dots ,\alpha_n) \in \mathbb{N}^n.
\end{equation}
\textbf{Classical case}.
In this case, we have $\mu_1=\dots =\mu_n=\mu$. Then by Proposition \ref{propo_lema} part $(i)$, we have $b^{-1}=\mu^{m}$ with $m\geq -1$. Then, it  follows from equation (\ref{eq1}) that
\begin{equation*}
c_{\alpha}^k \mu ^{\alpha_1+\dots  + \alpha_n}=c_{\alpha}^k \mu^{m+1},
\textrm{ where } \alpha =(\alpha_1, \alpha_2, \dots ,\alpha_n) \in \mathbb{N}^n.
\end{equation*}
If $c_{\alpha}^k \neq 0$, then
$\mu^{\alpha_1+\dots  + \alpha_n}=\mu^{m+1}$ with $m\geq -1$. Thus
$\alpha_1+\dots  + \alpha_n=m+1\geq 0$.
This shows that each $g_k$ is a homogeneous polynomial of degree $m+1$ and the foliation $\mathcal{F}$ is given by a polynomial vector field $g_1\frac{\partial}{\partial{z}_{1}}+\ldots+g_n\frac{\partial}{\partial{z}_{n}}$, where $g_{i}$ are homogeneous polynomial of degree $m+1\geq 0$, for all $1\leq i\leq n$.
\\
\textbf{Generic case.}
In this case,   the $\mu_i$ are not related. Then by Proposition \ref{propo_lema} part $(ii)$, we have
$b^{-1}=\mu_1 ^{d_1}\mu_2 ^{d_2}\dots \mu_n ^{d_n}$ such that there exists $d_{j_{0}}\geq -1$ and $d_{j}\geq 0$ for all $j\in\{1,\ldots,n\}\setminus\{j_{0}\}$. From the equation (\ref{eq1}), if $c_{\alpha}^k \neq 0$ then $b^{-1}= \mu_1 ^{\alpha_1} \mu_2 ^{\alpha_2}\dots  \mu_ n ^{\alpha_n} \mu_k ^{-1}$ and so that $\alpha_j=d_{j}$ for $j\neq k$ and $\alpha_k=d_k+1$. Hence for all $j=1,\ldots,n$ we get
\begin{equation}\label{equacaogenerica}
g_k(z_1,\dots ,z_n)=c^k z_1^{d_1}\dots z_k^{d_k+1}\dots z_n^{d_n},
\end{equation}
where $c^k=c^k _\alpha$, $\alpha=(d_1,d_2,\dots,d_k+1,\dots,d_n)$.
Since $\fol$ is nonsingular, there are two possibilities:
\begin{itemize}
\item [(i)] $d_{j_0}=-1$ and $d_j=0$ for $j\neq j_0$. In this case  $b=\mu_{j_0}$ and the foliation is induced by a vector field of the form $v=\frac{\partial}{\partial z_{j_0}}$.
\item [(ii)]$d_j=0$ for all $j=1,\dots,n$. In this case  $b=1$ and the linear vector field $v=c^1z_1 \frac{\partial}{\partial z_{1}}+\dots +c^n z_n\frac{\partial}{\partial z_{n}}$ with
$c^j\neq0\, \forall \, j=1,\dots ,n$, induces the foliation.
\end{itemize}
\textbf{Intermediary case.}
In this case, we have $\mu_1=\dots =\mu_r=\mu$. Then by Proposition \ref{propo_lema} part $(iii)$, we have
\begin{equation}\label{equa21}
b^{-1}=\mu^{d} \mu_{r+1}^{d_{r+1}}\dots \mu_n^{d_n},
\end{equation}
where  $d\geq -1$ and $d_j\geq 0$ for $j \geq r+1$
or $d\geq 0$  and there is $j_0\geq r+1 $ with $d_{j_0}\geq-1$ and $d_j\geq 0$ for all $j \in \{r+1, r+2,\dots ,n\}\setminus\{j_0\}$. From the equation (\ref{eq1}) we deduce that
\begin{center}
$\mu^{\alpha_1+\alpha_2+\dots +\alpha_r} \mu_{r+1}^{\alpha_{r+1}}\dots \mu_{n}^{\alpha_n}=\mu_k \mu^d\mu_{r+1} ^{d_{r+1}}\dots \mu_n^{d_n}$.
\end{center}
Since $\mu, \mu_{r+1},\dots ,\mu_n$ are not related we have

\begin{flushleft}
$g_k=z_{r+1}^{d_{r+1}}z_{r+2}^{d_{r+2}}\dots z_n ^{d_n} \sum\limits_ {\alpha \in \mathbb{N}^r} c_{\alpha} ^kz_1 ^{\alpha_1}\dots z_r ^{\alpha_r}
\textrm{ with } |\alpha|=d+1  \textrm{ and } 1\leq k\leq r,$\\
$
g_k=z_{r+1}^{d_{r+1}}z_{r+2}^{d_{r+2}}\dots z_k^{d_k+1}\dots  z_n ^{d_n} \sum\limits_ {\alpha \in \mathbb{N}^r} c^k _{\alpha} z_1 ^{\alpha_1}\dots z_r ^{\alpha_r} \textrm{ with } |\alpha|=d \textrm{ and } r+1\leq k\leq n.
$
\end{flushleft}
Since $\fol$ is nonsingular, there are three possibilities:

\begin{itemize}
\item [(i)] $d=-1$ and $d_j=0$ for all $j\geq r+1$. In this case $b=\mu=\mu_1=\dots=\mu_r$ and the vector field $$v=c^1 \frac{\partial}{\partial z_1}+\dots +c^r \frac{\partial}{\partial z_r},$$ for some $c^j\neq 0$, induces the foliation.

\item [(ii)] $d=0$ and $d_j=0$ for all $j\geq r+1$. So $b=1$ and
 \begin{center}
$v=g_1\frac{\partial}{\partial z_1}+\dots +g_r \frac{\partial}{\partial z_r}+
c^{r+1}z_{r+1}\frac{\partial}{\partial z_{r+1}}+\dots +c^{n}z_{n}\frac{\partial}{\partial z_{n}}$
\end{center}
with
\begin{center}
$g_k= \sum\limits_ {i=1} ^{r}  c^k _ {i} z_i $ para $1\leq k\leq r$
\end{center}
\item [(iii)] $d=0$, $d_{j_0}=-1$, for some $j_{0}\in \{r+1,\dots, n\}$ and $d_j=0$ for all $j\in \{r+1,\dots, n\}\setminus \{j_{0}\}$. In this case $b=\mu_{j_0}$ and $v=\frac{\partial}{\partial z_{j_0}}$.

\end{itemize}

\begin{re}
Notice that Theorem \ref{generic_theorem} follows  from equation (\ref{equacaogenerica}).
\end{re}

\section{One-codimensional holomorphic foliations }
In this section we will give a classification nonsingular holomorphic distributions and foliations of codimension-one on classical, generic or intermediary Hopf manifolds.
\begin{lema}\label{le1}
Let $X$ be a classical, generic or intermediary Hopf manifold of dimension $n\geq3$, and $L_a$ be a line bundle on $X$, with $a\in \mathbb{C}^*$. The following holds:
 \begin{enumerate}
\item[(i)] If $X$ is classical then $\dim\,H^0(X,\Omega_X ^1 \otimes L_a)>0$ if, and only if, $a=\mu^m$, where $m\in \mathbb{N}, \, m\geq1$.
\item[(ii)] If $X$ is generic then $\dim\, H^0(X,\Omega_X ^1 \otimes L_a)>0$ if, and only if, $$a=\mu_1 ^{m_1}\mu_2 ^{m_2}\dots \mu_n ^{m_n}$$ where $m_i\in \mathbb{N}$ and there exists $j_0\in \{1,\dots  ,n\}$, such that $m_{j_0}\geq 1$.
\item[(iii)] If $X$ is intermediary then  $\dim\,H^0(X,\Omega_X ^1 \otimes L_a)>0$ if, and only if, $$a=\mu_1^{m_1} \mu_2^{m_2}\dots \mu_n^{m_n}$$ with
$\mu_j \in \mathbb{N}$ for all $j=1,\dots ,n$, and
 $m_1+m_2+\dots +m_r\geq 1$, $m_j \geq 0$ for $j \geq r+1$ or \mbox{$m_1+m_2+\dots +m_r\geq 0$}, $m_j \geq 0$ for $j \geq r+1$, and there exists $j_0\geq r+1$ with $m_{j_0} \geq 1$.
\end{enumerate}

\end{lema}
\begin{proof}
According to Theorem \ref{teocoh} we have
$\dim H^0(X, \Omega_X ^1(L_a))=dim\,ker(p_0),$
where
$$p_0=a\cdot Id-f^*: H^0 (W, \Omega_W ^1)\rightarrow H^0 (W, \Omega_W ^1)$$
and $W=\mathbb{C}^{n}-\{0\}$. Let $\omega \in \Gamma(W,\Omega_W ^1)$, so that $\omega= \sum\limits _{i=1} ^{n} g_i dz_i$, applying Hartogs extension theorem, each $g_i$, $1\leq i\leq n$, can be represented by its Taylor series
\begin{center}
$g_i(z_1,z_2,\dots ,z_n)=\sum\limits _{\alpha \in \mathbb{N}^n} c_\alpha ^i z_1^{\alpha_1}z_2^{\alpha_2}\dots  z_n ^{\alpha_n}$.
\end{center}
Then $
p_0(\omega)=\sum\limits _{i=1} ^{n} \sum\limits _{\alpha \in \mathbb{N}^n}c_\alpha ^i (a-\mu_1 ^{\alpha_1} \dots  \mu_n ^{\alpha_n}  \mu_i) z_1^{\alpha_1}z_2^{\alpha_2}\dots  z_n ^{\alpha_n}dz_i \,.
$
Hence
\begin{center}
$p_0(\omega)=0 \Leftrightarrow \forall \alpha \in \mathbb{N}^n, \,i \in \{1,\dots ,n\}$ this implies that $c_\alpha ^i=0$ or
$a=\mu_1 ^{\alpha_1} \dots \mu_n ^{\alpha_n}  \mu_i$.
\end{center}

If $X$ is a classical Hopf, that is, $\mu_1=\dots =\mu_n=\mu$ then
$p_0(\omega)=0 \Leftrightarrow \forall \alpha \in \mathbb{N}^n \textrm{ and } \,i \in \{1,\dots ,n\} \textrm{ so that } c_\alpha ^i=0$ or $a=\mu^{\alpha_1+\dots +\alpha_n+1}$.
Therefore,
$dim(ker p_0)>0 \Leftrightarrow a=\mu^m \textrm{ with } m \in \mathbb{N}, \, m\geq 1$.
If $X$ is generic, since  there are not relations between the $\mu_i$'s, we have
$\dim(ker p_0)>0 \Leftrightarrow a=\mu_1^{m_1}\dots \mu_n ^{m_n},$ where $\mu_i \in \mathbb{N}, \, m_i\geq 0$ for all $i\in \{1,\dots ,n\}$ and there exists $i_0 \in \{1,\dots ,n\}$ such that $m_{i_0}\geq 1$.
 Finally, if $X$ is intermediary then $\mu_1=\dots =\mu_r$ so that
$$p_0(\omega)=0 \Leftrightarrow \forall \alpha \in \mathbb{N}^n, \,i \in \{1,\dots ,n\}$$
 and hence $c_\alpha ^i=0$ or $a=\mu_1^{\alpha_1+\dots +\alpha_r}\mu_{r+1}^{\alpha_{r+1}}\dots \mu_n^{\alpha_n}\mu_i$. As there are not relations between $\mu_1,\mu_{r+1},\dots ,\mu_n$ we have
$$\dim(ker\,p_0)>0 \Leftrightarrow a=\mu_1^{m_1+\dots +m_r}\mu_{r+1}^{m_{r+1}}\dots \mu_n ^{m_n},$$
with $m_1+\dots +m_r\geq 1$  and $m_j\geq 0\,\, \forall j\geq r+1$, or $m_1+\dots +m_r\geq 0$, $m_j\geq 0$ $\forall$ $j\geq r+1$  and there exists $j_0\geq r+1$ with $m_{j_0}\geq 1$.\end{proof}
\par As consequence of above lemma, we obtain the following proposition.

\begin{proposition}\label{teocod}
Let $X$ be a Hopf manifold of dimension $n\geq 3$ and let $\mathcal{F}$ be a nonsingular codimension-one holomorphic distribution on $X$ with
$\mathcal{N}_{\mathcal{F}}=L_{b^{-1}}$.
The following holds:
\begin{enumerate}
\item[(i)] If $X$ is classical then $b^{-1}=\mu^m$, com $m \in \mathbb{N}$ and $m\geq 1$.
\item[(ii)] If $X$ is generic then $b^{-1}=\mu_1 ^{m_1} \mu_2^{m_2}\dots \mu_n ^{m_n}$ where $\mu_j \in \mathbb{N}$ for all $j=1,\dots ,n$, e $\mu_1+\dots +\mu_n\geq 1$.
\item[(iii)] If $X$ is intermediary then $b^{-1}=\mu_1^{m_1} \mu_2^{m_2}\dots \mu_n^{m_n}$ where $\mu_j \in \mathbb{N}$ for all $j=1,\dots , n$, and $m_1+m_2+\dots +m_r\geq 1$, $m_j \geq 0$ for $j \geq r+1$ or $m_1+m_2+\dots +m_r\geq 0$ e $m_j \geq 0$ for all $j \geq r+1$ and there exists $j_0\geq r+1$ with $m_{j_0} \geq 1$.
\end{enumerate}
\end{proposition}

\subsection{Proof  of Theorem \ref{teo1}}
By construction,  we have that
$$\Omega_X ^1 \otimes L_{b^{-1}} \cong  (W \times \mathbb{C} ^n )/(f \times A^{-1} b^{-1}),$$ where  $
W=\mathbb{C}^n - \{0\}$ and $
A^{-1} =\left(\mu_1^{-1} ,\dots,\mu_n^{-1}  \right)
.$
Thus, the holomorphic section $s \in H^0(X, \Omega_X ^1 \otimes L_{b^{-1}})$ correspond to a section $\tilde{s} \in H^0(W, \mathcal{O}_W ^n)$, say $\tilde{s}=(g_1,\dots ,g_n)$, such that
$g_k \in \mathcal{O}_W$ satisfies
$g_k(\mu_1 z_1,\dots ,\mu_n z_n)= \mu_k^{-1}b^{-1} g_k(z_1,\dots ,z_n)$, for all $k=1,\dots ,n.$
By Hartog's extension theorem, $\tilde{s}$ can be represented by its Taylor series
\begin{center}
$g_k(z_1,\dots ,z_n)=\sum\limits _{\alpha \in \mathbb{N}^n} c_\alpha ^k z_1^{\alpha_1}\dots  z_n ^{\alpha_n}$, where $\alpha=(\alpha_1,\alpha_2,\dots ,\alpha_n) \in \mathbb{N}^n$ .
\end{center}
Then
\begin{equation}\label{equa6}
c_{\alpha}^k \mu_1 ^{\alpha_1} \mu_2 ^{\alpha_2}\dots  \mu_ n ^{\alpha_n}=c_{\alpha}^k \mu_k^{-1} b^{-1},
 \textrm{ where } \alpha =(\alpha_1, \alpha_2, \dots ,\alpha_n)\in \mathbb{N}^n.
\end{equation}
\textbf{The classical case.} In this case $\mu_1=\dots = \mu_n=\mu$. Proposition \ref{teocod} part $(i)$ implies that
 $b^{-1}=\mu^{m}$ for some
$m\geq 1$.
Therefore
$
c_{\alpha}^{k}\mu^{|\alpha|}=c_{\alpha}^k\mu^{-1}\mu^{m} \textrm{ where } |\alpha|=\alpha_1+\dots +\alpha_n.
$
Hence, if $ c_{\alpha}^{k}\neq 0$ then $|\alpha|=m-1$. It follows that  each $g_k$ is a homogeneous polynomial of degree $m-1$.
\\
\textbf{Generic case.} If $X$ is generic, then by Proposition \ref{teocod} part $(ii)$ we have
$$b^{-1}=\mu_1 ^{m_1} \mu_2^{m_2}\dots \mu_n ^{m_n},$$
where $\mu_j \in \mathbb{N}$, $\mu_j\geq 0$ for all $j \in \{1,\dots ,n\} $ and there exists $j_0 \in \{1,\dots ,n\}$ such that  $m_{j_0}\geq 1$. Then from (\ref{equa6}) we get
\begin{center}
$c_{\alpha}^k \mu_1 ^{\alpha_1} \mu_2 ^{\alpha_2}\dots  \mu_ n ^{\alpha_n}=c_{\alpha}^k \mu_k^{-1} \mu_1 ^{m_1} \mu_2^{m_2}\dots \mu_n ^{m_n} , \textrm{ where } \alpha =(\alpha_1, \alpha_2, \dots ,\alpha_n)\in \mathbb{N}^n.$
\end{center}
Hence for each $k=1,\dots ,n$ we have
\begin{equation}\label{generica}
g_k(z_1,\dots  ,z_n)=c^k z_{1}^{m_1}z_2 ^{m_2}\dots  z_{k}^{m_k -1}\dots  z_n^{m_n} .
\end{equation}
Since $\fol$ is nonsingular, we get that $m_{j_0}=1$ for some $j_{0}\in \{1,\dots, n\}$ and $m_j=0$ for all $j\in\{1,\dots, n\}\setminus\{j_{0}\}$. So that we have $b^{-1}=\mu_{j_0}$, $g_{j_{0}}$ is a constant
and $g_{j}=0$ for all $j\in\{1,\dots, n\}\setminus\{j_{0}\}$.
\\
\textbf{Intermediary case.}
In this case, we have $\mu_1=\dots =\mu_r=\mu$. Moreover, Proposition \ref{teocod} part $(iii)$, implies that
$b^{-1}=\mu_1^{m_1} \mu_2^{m_2}\dots \mu_n^{m_n},$
where \mbox{$m_1+m_2+\dots +m_r\geq 1$}
and $m_j \geq 0$ for $j \geq r+1$ or \mbox{$m_1+m_2+\dots +m_r\geq 0$} and $m_j \geq 0$ for all $j \geq r+1$ and there is $j_0\geq r+1$ with  $m_{j_0} \geq 1$.
 Then from (\ref{equa6}) we get
\begin{center}
$c_{\alpha}^k \mu_1 ^{\alpha_1+\dots +\alpha_r}\mu_{r+1}^{\alpha_{r+1}}\dots  \mu_ n ^{\alpha_n}=c_{\alpha}^k \mu_k^{-1}\mu_1^{m_1+\dots +m_r}\mu_{r+1}^{m_{r+1}}\dots \mu_n ^{m_n}$,
\end{center}
where  $\alpha =(\alpha_1, \alpha_2, \dots ,\alpha_n).$
Since there are no relations between $\mu, \mu_{r+1},\dots ,\mu_n$, we have that for each $k=1,\dots ,r$,
\begin{center}
 $g_k(z_1,\dots ,z_n)=z_{r+1}^{m_{r+1}}\dots  z_{n}^{m_n} \sum\limits_{\alpha} c_{\alpha} ^k z_{1} ^{\alpha_1}\dots  z_r ^{\alpha_r} $, \end{center}
$\textrm{ where }  \alpha_1+\dots +\alpha_r=m_1+\dots +m_r -1\geq 0,$ and for each $k=r+1,\dots ,n$,
\begin{center}
 $g_k(z_1,\dots ,z_n)= z_{r+1}^{m_{r+1}}\dots  z_k^{m_k-1}\dots  z_{n}^{m_n}\sum\limits_{\alpha} c_{\alpha} ^k z_{1} ^{\alpha_1}\dots  z_r ^{\alpha_r},$\end{center}
 $
\textrm{ where }  \alpha_1+\dots +\alpha_r=m_1+\dots +m_r \geq 1.$
Since $\fol$ is nonsingular, we have the following possibilities:
\begin{itemize}
\item [(i)] $b =\mu_1 ^{-1}=\dots =\mu_r ^{-1}$. In this case the foliation is induced by a holomorphic 1-form of the type $\omega=c^1dz_1+\dots +c^rdz_r$ , with $c^{j}\neq 0$ for some $j=1,\dots, r$.
\item [(ii)] $b=\mu_j ^{-1}$ for some $j\in \{ r+1,\dots,n\}$. In this case, the foliation is induced by a 1-form of the type $\omega=dz_{j}$ .
\end{itemize}
\subsection{Proof of Theorem \ref{prop6}}
The equation (\ref{generica}) shows that every codimension-one distribution on a generic Hopf manifold is induced by a monomial  1-form
of the type
\begin{eqnarray*} &&
\omega=\sum\limits _{i=1} ^{n} g_i dz_i, \textrm{ where } g_i(z_1,\dots , z_n)=c_i z_{1}^{m_1}z_2 ^{m_2}\dots z_{i}^{m_i -1}\dots z_n^{m_n},
\textrm{ with } m_j\geq 0,
\end{eqnarray*}
Now, we show that $\omega \wedge d\omega=0$. Calculating $d\omega$ we obtain
\begin{equation*}
d\omega= \sum\limits _{i,j =1} ^{n} c_i m_j  z_{1}^{m_1}z_2 ^{m_2}\dots z_i ^{m_i-1} \dots z_j ^{m_j-1}\dots z_n^{m_n} dz_j\wedge dz_i.
\end{equation*}
Define
$\nu_{ijk}=c_i m_j c_k=c_k m_j c_i=\nu_{kji}$ and $$z_{ijk}= z_{1}^{2m_1}z_2 ^{2m_2}\dots z_k ^{2m_k -1}\dots z_j ^{2m_j-1}\dots z_i ^{2m_i-1}\dots z_n^{2m_n}.$$ Then, we get
\begin{eqnarray*}&&
w\wedge dw = \sum\limits _{k<j<i} (\nu_{kji}- \nu_{kij}- \nu_{jki}+\nu_{jik} - \nu_{ijk} + \nu_{ikj}) z_{ijk} dz_k \wedge dz_j\wedge dz_i =0
\end{eqnarray*}
\section{Singular foliations on Hopf manifolds}
In this section we study singular holomorphic foliations (and distributions) on Hopf manifolds.
Firstly,  we prove  Theorem \ref{singholomorphic}.
\subsection{ Proof of Theorem \ref{singholomorphic}}
Let $\xi\in H^0(X, TX\otimes L)$ be a section inducing $\F$.
Suppose by contradiction that $Sing(\xi)$  is nonempty  and  has only isolated points. By Baum-Bott Theorem
\cite{Baum} we have
\begin{equation}\label{eq7}
c_n(TX\otimes L)=\sum\limits_{\{p: \xi (p)=0\}} \mu_p(\xi)>0
\end{equation}
where $\mu_p(\xi)$ denotes the Milnor number of $\xi$ at $p$, and $c_n$ the top Chern class.
 For the one-codimensional case the idea of the proof is the same.
 In this case we use the Baum-Bott type Theorem  due to T. Izawa \cite{izawa}.
On the other hand,
\begin{equation}\label{eq3}
c_n(TX)= \sum\limits_{p,q} (-1)^{p+q} h^{p,q},
\end{equation}  where $h^{p,q} =\dim H^q(X, \Omega^p)$ is the Hodge number.
Mall in \cite{Ma} showed that
\begin{equation}\label{eq4}
h^{0,0}=h^{0,1}=h^{n,n}=h^{n,n-1}=1 \textrm{ and } h^{p,q}=0 \textrm{ in all other cases}.
\end{equation}
Then from (\ref{eq3}) and (\ref{eq4}) we have that $c_n(TX)=0$. Furthermore, since $X$ is diffeomorphic to $S^{2n-1}\times S^1$, by K$\ddot{\textrm{u}}$nneth formula we get
$H^2(X,\mathbb{Z})=0$.
In particular, the first Chern class $c_1(L)\in H^2(X,\mathbb{Z})$ vanishes. Then
\begin{center}
$c_n(TX\otimes L)=\sum\limits_{j=0}^{n} c_j(TX)c_1(L)^{n-j}= c_n(TX)=0$
\end{center}
which contradicts the equation (\ref{eq7}).

\subsection{Proof of Theorem \ref{Teo_brunella}}

We prove the  Brunella type  alternative  for holomorphic foliations on classical Hopf manifolds.
Let $\omega$ be a $1$-form on $\mathbb{C}^n-\{0\}$ inducing $\F$. Then we get
$\omega=g_1dz_1+\dots +g_n dz_n,$ where  $g_i$ are  homogeneous
polynomials of the same degree  $k$ for all $1\leq i\leq n$. Let $\alpha: {X}\rightarrow \mathbb{P}^{n-1}, \,\, \alpha(z)=[z]$, the natural morphism. Then is either  $\F$  is transversal to a generic fiber of $\pi$ or  $\F$ is tangent to fibration $\pi$.

In the first case, the contraction $f:=i_R \omega=\sum\limits _{i=1} ^{n} z_i g_i$, of the $\omega$ by radial vector field  $R=\sum\limits _{i=1} ^{n} z_i \frac{\partial}{\partial z_i}$, is a not identically zero analytic function on $X$.
We will show that the analytic hypersurface $\{f=0\}$ is invariant by ${\F}$. In fact, the integrability condition, $\omega \wedge d \omega=0$, implies that $ i_R\omega \wedge d\omega + \omega \wedge i_R d\omega=0$.
On the other hand,
\begin{eqnarray*}
i_R d\omega +d(i_R \omega)&=&\sum _{i=1} ^{n} i_R(dg_i \wedge dz_i)+ d\left(\sum\limits _{i=1} ^{n} z_i g_i\right)=(k+1)\omega.
\end{eqnarray*}
Taking the exterior product with $ \omega $, we get $\omega \wedge i_R d\omega  + \omega \wedge d(i_R\omega )=0$ and
by using this equation and  $i_R\omega \wedge d\omega + \omega \wedge i_R d\omega=0$,  we obtain
$\omega \wedge d(i_R\omega)=(i_R\omega ) d\omega.$
Thus  $\{i_R\omega =0\}$ is invariant by $\F$. On the other hand, if   $\F$ is tangent to the fibration $\pi$, then $\F$ is
subfoliated by  elliptic  curves.
\\
\\
\noindent
\textbf{Acknowlegments.}  We are grateful to Bruno Sc\'ardua and J. Omegar Calvo-Andrade
for interesting conversations and corrections. We also would like to thank the referee for pointing out corrections and suggestions.

\end{document}